\documentclass[12pt]{amsart}

\usepackage[utf8]{inputenc}
\usepackage{amsmath}
\usepackage{amsthm}
\usepackage{amsfonts}
\usepackage{amssymb}
\usepackage{tabu}
\usepackage{latexsym} 
\usepackage[margin=1in]{geometry}
\usepackage{enumerate} 
\usepackage[colorlinks=true]{hyperref}
\usepackage{mathrsfs}
\usepackage{tikz}
\usepackage[parfill]{parskip}
\usetikzlibrary{matrix,arrows,decorations.pathmorphing}
\usepackage{tikz-cd}
\usepackage[all]{xy}
\usepackage{setspace}
\usepackage{tabularx}
\usepackage{xcolor}
\usepackage{framed}

\newcommand{\Ext}{\mathrm{Ext}}

\newcommand{\Hom}{\mathrm{Hom}}
\newcommand{\End}{\mathrm{End}}
\newcommand{\sHom}{\underline{\mathrm{Hom}}}
\newcommand{\sEnd}{\underline{\mathrm{End}}}

\newcommand{\ann}{\mathrm{ann}}
\newcommand{\sann}{\underline{\mathrm{ann}}}

\newcommand\ca{\mathrm{ca}}
\newcommand\co{\mathrm{co}}
\newcommand{\m}{\mathfrak{m}}

\newcommand{\module}{\mathrm{mod}}

\newcommand{\CM}{\mathrm{CM}}

\newcommand{\reflexive}{\mathrm{ref}}

\newcommand{\CC}{\mathcal{C}}

\newcommand{\XX}{\mathcal{X}}

\newcommand{\sing}{\mathrm{Sing}}

\newtheorem{theorem}{Theorem}[section]
\newtheorem*{theorem*}{Theorem}
\newtheorem{lemma}[theorem]{Lemma}
\newtheorem*{lemma*}{Lemma}
\newtheorem{proposition}[theorem]{Proposition}
\newtheorem*{proposition*}{Proposition}
\newtheorem{corollary}[theorem]{Corollary}
\newtheorem*{corollary*}{Corollary}
\newtheorem{question}[theorem]{Question}
\theoremstyle{definition}
\newtheorem{definition}[theorem]{Definition}
\newtheorem*{definition*}{Definition}
\newtheorem{example}[theorem]{Example}
\newtheorem{remark}[theorem]{Remark}

\newtheorem{introtheorem}{Theorem}

\numberwithin{equation}{theorem}
\def\CC{\mathbb{C}}

\DeclareMathOperator{\tr}{\mathrm{tr}}

\usepackage[normalem]{ulem}

\title{Annihilation of cohomology over one dimensional almost Gorenstein rings}
\author[\"{O}.~Esentepe]{\"{O}zg\"{u}r Esentepe}
\address{Institut für Mathematik und Wissenschaftliches Rechnen, Universität Graz,
Heinrichstraße 36, 8010 Graz, Austria}
\email{ozgur.esentepe@uni-graz.at}
\urladdr{https://www.sntp.ca}


\begin{document}

\begin{abstract}
   Given a Cohen-Macaulay local ring, the cohomology annihilator ideal and the annihilator of the stable category of maximal Cohen-Macaulay modules are two ideals closely related both with each other and the singularities of the ring. Kimura recently showed that the two ideals agree up to radicals. In this article, we give a sufficient condition for the two ideals to be equal. As an application, we show that the cohomology annihilator ideal of a one dimensional analytically unramified almost Gorenstein complete local ring agrees with the conductor ideal. 
\end{abstract}

\maketitle
\thispagestyle{empty}

\section{Introduction}

Given a field $k$, the power series ring $k[\![x]\!]$ in one variable and a finitely generated $k[\![x]\!]$-module $M$ generated by the elements $m_1, \ldots, m_n$; one can construct a surjective $k[\![x]\!]$-linear map $k[\![x]\!]^{\oplus n} \to M$ by sending the standard basis vector $e_i$ to $m_i$. The kernel of this map encodes all the relations between the generators $m_1, \ldots, m_n$. It was known at the end of the nineteenth century that the kernel of this map will always be a free $k[\![x]\!]$-module. In fact, for polynomial rings of several variables Hilbert proved the analogous result for higher syzygies in \cite{Hilbert} and this has started the interplay between homological algebra and singularities. Today, one can use the language of homological algebra to state that any finitely generated module over a polynomial ring $S$ in $d$ variables has projective dimension at most $d$. One can also use the language of derived functors and say that for any two finitely generated $S$-modules $M$ and $N$, we have $\Ext_S^n(M,N) = 0$ for any $n > d$. Over the last century, this idea has been generalised vastly and became a fundamental tool in commutative algebra and algebraic geometry. Now, we know, for almost a century, that a commutative Noetherian local ring is regular if and only if it has finite global dimension.

When a local ring $R$ is not regular, for any positive integer $n$, we can find finitely generated $R$-modules $M,N$ such that $\Ext_R^n(M,N) \neq 0$. The structure of these nonzero Ext-modules can give us information about the singular points on the corresponding algebraic variety. To this end, one studies the annihilators of these Ext-modules. A systematic approach to this problem was initiated by Iyengar and Takahashi in \cite{iyengar-takahashi} who introduced the notion of a  \textit{cohomology annihilator ideal} $\ca(R)$ consisting of the ring elements which uniformly annihilate all high enough Ext-modules (see Definition \ref{definition-cohomology-annihilator} for a more precise formulation). One can now use this language to say that $R$ has finite global dimension if and only if $\ca(R) = R$ \cite[Example 2.5]{iyengar-takahashi}. They proved that the vanishing locus of this ideal is the singular locus of $R$ for a large class of rings. It is not difficult to find elements in this ideal. We present some examples.
\begin{enumerate}
    \item If $R$ is a hypersurface ring $R = \CC[\![x_1, \ldots, x_n]\!]/(f)$ defined by a polynomial $f$ in the maximal ideal $(x_1, \ldots, x_n)$. Then, the partial derivatives of $f$ with respect to the variables $n$ lie in the cohomology annihilator ideal of $R$. More generally, if we consider a regular sequence $(f_1, \ldots, f_m)$ in $(x_1, \ldots, x_m)$ and put $R$ to be the complete intersection ring $\CC[\![x_1, \ldots, x_n]\!]/(f_1, \ldots, f_m)$, then the cohomology annihilator ideal of $R$ contains the Jacobian ideal of $R$ which is generated by the maximal minors of the Jacobian matrix whose entries are given by the partial derivatives $\partial f_j/\partial x_i$. This appears in the work of Buchweitz on Maximal Cohen-Macaulay modules and Tate cohomology \cite[Section 7]{Buchweitz}. We refer to Appendix C.6 of \textit{op. cit.} for recent developments.
    \item Let $R$ be a one dimensional reduced complete Noetherian local ring, $Q(R)$ be its total quotient ring and $\overline{R}$ be the integral closure of $R$ inside $Q(R)$. The conductor ideal $\co(R)$ of $R$ is defined as $\co(R) = \{ r\in R \colon r \overline{R} \subseteq R\} $. Wang proved in \cite{Wang} that the conductor ideal, in this case, is contained in the cohomology annihilator ideal.     
\end{enumerate}
However, it turns out to be a difficult task to determine this ideal completely. We have only a few classes of rings where we have a complete description of the cohomology annihilator ideal.
\begin{enumerate}
    \item Let $R$ be a one dimensional reduced complete Noetherian local ring with maximal ideal $\m$ and assume that $R$ is analytically unramified and that there is an isomorphism $\End_R(\m) \cong \overline{R}$. Such rings were studied by Faber in \cite{Faber} who called them \textit{one-step normal}. They also appear as \textit{far-flung Gorenstein nearly Gorenstein} rings of \cite{Herzog-Kumashiro-Stamate} and one dimensional \textit{Ulrich split} rings of \cite{Dao-Dey-Dutta} in the literature. Faber proved that if $R$ is a singular one step normal ring, then the conductor ideal equals the maximal ideal. Combining this with Wang's result, one has a complete description of the cohomology annihilator ideal: $\ca(R) = \co(R)$. 
    \item One can extend this equality to the Gorenstein case. In particular, it is proved in \cite[Theorem 5.10]{Esentepe} that if $R$ is an analytically unramified one dimensional Gorenstein local ring, then the equality $\ca(R) = \co(R)$ still holds. 
\end{enumerate}
In this paper, we prove that the equality of the conductor and the cohomology annihilator extends to a larger class of rings which are close to being Gorenstein but not necessarily Gorenstein. To do so, we first relate the cohomology annihilator to the annihilator of the stable category of maximal Cohen-Macaulay modules (see Definition \ref{general-definitions}(3), Remark \ref{cohomology-annihilator-remarks}(4, 5)).
\begin{introtheorem}[Theorem \ref{main-theorem}]\label{introtheorem-a}
 Let $R$ be a Cohen-Macaulay local ring with canonical module $\omega$. If $\Omega \CM^\times(R)$ is closed under the canonical duality functor $D(-)$, then there exists an equality 
    \begin{align*}
    \ca(R) = \sann_R (\CM(R)).
    \end{align*}
\end{introtheorem}
Then, we apply this theorem to the class of almost Gorenstein rings.
\begin{introtheorem}[Theorem \ref{theorem-cohomology-annihilator-of-an-almost-Gorenstein-ring}]\label{introtheorem-b}
     Let $R$ be an analytically unramified one dimensional almost Gorenstein complete local ring. Then, we have an equality 
    \begin{align*}
    \ca(R) = \co(R).
    \end{align*}
\end{introtheorem}
The paper structured as follows. In Section 2, we give preliminaries regarding stable annihilator ideals and prove Theorem \ref{introtheorem-a}. Then, we look at the hypothesis of this theorem: when is $\Omega \CM^\times(R)$ closed under the canonical duality functor? When the ring has minimal multiplicity, Kobayashi-Takahashi studied this condition in \cite{Kobayashi-Takahashi}. We both recall their findings and also slightly improve their theorems. We also relate our findings to \textit{cocohomological} annihilators of a module which were introduced recently by Dey-Liu-Mifune-Otake \cite{Dey-Liu-Mifune-Otake}.

In Section 3, we consider modules in $\Omega \CM(R)$ whose canonical dual also belong to $\Omega \CM(R)$. The main actors in this section are trace ideals and $I$-Ulrich modules introduced by Dao-Maitra-Sridhar \cite{Dao-Maitra-Sridhar}. We show that a module $M \in \Omega \CM(R)$ has its canonical dual $DM$ in $\Omega \CM(R)$ if and only if its trace ideal $\tr(M)$ is contained in the canonical trace ideal $\tr(\omega)$ provided that $R$ has canonical reduction number 2 \cite{Kumashiro}. We conclude the paper by proving Theorem \ref{introtheorem-b}.

\section{Stable annihilators and the main theorem}

In this section, we give some preliminaries about annihilators of Ext modules and prove our main theorem. Throughout, we work with commutative Noetherian local rings and all modules are assumed to be finitely generated. For a commutative ring $R$, we denote by $\module R$ the category of finitely generated $R$-modules. When $R$ is a Cohen-Macaulay local ring, $\CM(R)$ denotes the category of maximal Cohen-Macaulay modules. We start with some general definitions.

\begin{definition}\label{general-definitions}
    Let $R$ be a commutative Noetherian local ring and $M,N$ be an $R$-modules. 
    \begin{enumerate}
        \item We denote by $D(-)$ the canonical dual functor $\Hom_R(-,\omega)$ when $R$ is a Cohen-Macaulay local ring with canonical module $\omega$. It is a duality when restricted to maximal Cohen-Macaulay modules and we call $DM$ the \textit{canonical dual} of $M$.
        \item We denote by $P(M,N)$ the set of all morphisms from $M$ to $N$ that factors through a projective $R$-module. This is a submodule of $\Hom_R(M,N)$. We denote the corresponding quotient module by $\sHom_R(M,N)$. When $M = N$, we use the notation $\sEnd_R(M)$.
        \item We denote by $\sann_R(M)$ the annihilator of the $R$-module $\sEnd_R(M)$ and call it the \textit{stable annihilator} of $M$. A ring element $r \in R$ belongs to $\sann_R M$ if and only if we have a commutative diagram
            \begin{align*}
                \xymatrix{
                M \ar^{r}[rr] \ar[dr]  && M \\
                &P \ar[ur]&
                }
            \end{align*}
            with $P$ a projective module where the horizontal arrow stands for multiplication by the element $r$. For a subcategory $\XX$ of $\module R$, we denote by $\sann_R(\XX)$ the intersection 
            \begin{align*}
            \sann_R(\XX) = \bigcap_{M \in \XX} \sann_R(M).
            \end{align*}
        \item Given a projective $R$-module $P$ and a surjective morphism $P \to M \to 0$, we call the kernel a \textit{syzygy} of $M$ and denote it by $\Omega_R M$. This is uniquely determined up to projective summands by Schanuel's Lemma. When it is clear from the context, we drop the subscript and use $\Omega M$ to denote a syzygy. We define the $n$th syzygy $\Omega^n M$ inductively. For a subcategory $\XX$ of $\module R$, we denote by $\Omega^n \XX$ the category consisting of $n$th syzygies of modules in $\XX$. 
    \end{enumerate}
\end{definition}
We now introduce cohomology annihilators. The definition is due to Iyengar and Takahashi \cite{iyengar-takahashi}.
\begin{definition}\label{definition-cohomology-annihilator}
    Let $R$ be a commutative Noetherian local ring.
    \begin{enumerate}
        \item For any positive integer $n$, define the \textit{nth cohomology annihilator ideal} as
        \begin{align*}
        \ca^n(R):= \{r \in R \colon r \Ext_R^{\geq n}(M,N) = 0 \text{ for all } M,N \in \module R\}.
        \end{align*}
        \item The \textit{cohomology annihilator ideal} $\ca(R)$ of $R$ is defined as the union 
        \begin{align*}
            \ca(R) := \bigcup_{n \in \mathbb{N} } \ca^n(R)
        \end{align*}    
    \end{enumerate}
\end{definition}
\begin{remark}\label{cohomology-annihilator-remarks}
    Let $R$ be a commutative Noetherian local ring and $M$ a finitely generated $R$-module.
    \begin{enumerate}
        \item The isomorphism $\Ext_R^{n+1}(M,N) \cong \Ext_R^n(\Omega M, N)$ gives us an inclusion of ideals $\ca^n(R) \subseteq \ca^{n+1}(R)$ for any $n \geq 1$. As $R$ is Noetherian, this implies that for $n \gg 0$, we have an equality $\ca(R) = \ca^n(R)$.  
        \item There are equalities 
        \begin{align*}
        \bigcap_{X \in \module R} \ann_R \Ext_R^{\geq n+1}(M,N) &= \ann_R\Ext_R^{\geq n+1}(M, \Omega^{n+1}M) \\
        &= \ann_R\Ext_R^1(\Omega^n M, \Omega^{n+1}M)
        \end{align*}
        for any nonnegative integer $n$. In particular, for $n = 0$, we have
        \begin{align*}
        \ann_R \Ext_R^1(M, \Omega M) = \bigcap_{X \in \module R} \ann_R \Ext_R^{n \geq 1}(M,X).
        \end{align*}
        This intersection also coincides with the stable annihilator $\sann_R(M)$. A proof of this can be found in \cite[Section 2]{iyengar-takahashi}.
        \item A combination of these remarks tells us that for $n \gg 0$, there is an equality 
        \begin{align*}
        \ca(R) = \sann_R(\Omega^n \module R).
        \end{align*}
        \item Assume that $R$ is a Cohen-Macaulay local ring of Krull dimension $d$. Then, there are inclusions 
        \begin{align*}
        \Omega^d \CM(R) \subseteq \Omega^d \module R \subseteq \CM(R)
        \end{align*}
        which yields, for any $n \geq 0$, inclusions 
        \begin{align*}
        \sann_R(\Omega^{n} \CM(R)) \subseteq \sann_R(\Omega^{n+d}\module R) \subseteq \sann_R(\Omega^{n+d} \CM(R)).
        \end{align*}
        Therefore, one has an equality 
        \begin{align*}
        \ca(R) = \sann_R(\Omega^n \CM(R))
        \end{align*}
        for $n \gg 0$.
        \item When $R$ is Gorenstein, there is an equality $\ca(R) = \sann_R(\CM(R))$ \cite[Lemma 2.3]{Esentepe}. Kimura showed in \cite[Theorem 1.1]{Kimura} that an equality 
        \begin{align*}
        \sqrt{\ca(R)} = \sqrt{\sann_R(\CM(R))} = \bigcap_{p \in \sing(\hat{R})} (p \cap R)
        \end{align*}
        holds for any Cohen-Macaulay local ring where $\sing(\hat{R})$ denotes the singular locus of the completion $\hat{R}$.
    \end{enumerate}
\end{remark}
Our goal is to give a sufficient condition for the equality $\ca(R) = \sann_R(\CM(R))$ to hold. We rely heavily on the following lemma.
\begin{lemma}\label{main-lemma}
    Let $R$ be a Cohen-Macaulay local ring with canonical module $\omega$ and $M$ a maximal Cohen-Macaulay $R$-module. Then, there are inclusions 
    \begin{align*}
    \sann_R(D\Omega M) \subseteq \sann_R(M) \subseteq \sann_R(\Omega M).
    \end{align*}
\end{lemma}
\begin{proof}
    By \ref{cohomology-annihilator-remarks}(2), we know that there is an inclusion 
    \begin{align*}
    \sann_R(D\Omega M) \subseteq \ann_R \Ext_R^1(D\Omega M, DM).
    \end{align*}
    On the other hand, the canonical dual gives an isomorphism 
    \begin{align*}
        \ann_R \Ext_R^1(D\Omega M, DM) = \ann_R \Ext_R^1(M, \Omega M)
    \end{align*}
    and again by Remark \ref{cohomology-annihilator-remarks}(2), the right hand side equals $\sann_R(M)$. This shows the first inclusion. The second inclusion is standard and uses similar techniques but we include a proof for completeness: 
    \begin{align*}
    \sann_R(M) \subseteq \ann_R\Ext_R^2(M, \Omega^2M) = \sann_R\Ext_R^1(\Omega M, \Omega^2 M) = \sann_R(\Omega M).
    \end{align*}
\end{proof}
The following proposition follows from Lemma \ref{main-lemma}.
\begin{proposition}\label{main-proposition}
    Let $R$ be a Cohen-Macaulay local ring with canonical module $\omega$ and $M$ a maximal Cohen-Macaulay $R$-module. If $D\Omega M \in \Omega \CM(R)$, then there is an equality 
    \begin{align*}
    \sann_R(M) = \sann_R(\Omega M).
    \end{align*}
\end{proposition}
\begin{proof}
    By our assumption, there exists a maximal Cohen-Macaulay $R$-module $X$ such that $D\Omega M = \Omega X$. This gives us $\Omega M = D\Omega X$. Applying Lemma \ref{main-lemma} to $X$, we get 
    \begin{align*}
    \sann_R(\Omega M) = \sann_R(D \Omega X) \subseteq \sann_R(M)
    \end{align*}
    which gives us the equality as we always have the inclusion $\sann_R(M) \subseteq \sann_R(\Omega M)$.
\end{proof}
Now, we are ready to prove our main theorem. The canonical module $\omega$ (which is the dual of the regular module $R$) never belongs to $\Omega \CM(R)$ unless $R$ is Gorenstein. Therefore, we need the category $\Omega \CM^{\times}(R)$ of syzygies of maximal Cohen-Macaulay modules without free summands in what follows.
\begin{theorem}\label{main-theorem}
    Let $R$ be a Cohen-Macaulay local ring with canonical module $\omega$. If $\Omega \CM^\times(R)$ is closed under the canonical duality functor $D(-)$, then there exists an equality 
    \begin{align*}
    \ca(R) = \sann_R (\CM(R)).
    \end{align*}
\end{theorem}
\begin{proof}
    By Remark \ref{cohomology-annihilator-remarks}(4), we know that for $n \gg 0$ one has an equality 
    \begin{align*}
        \ca(R) = \sann_R(\Omega^n \CM(R)).
    \end{align*}
    We will prove that $\sann_R(\Omega^n \CM(R)) = \sann_R(\CM(R))$ for any $n \geq 0$ and this will give us the desired equality. However, we see from Proposition \ref{main-proposition} that $\sann_R(M) = \sann_R(\Omega M)$ for any maximal Cohen-Macaulay module $M$. Hence, the ideals that appear in the intersection $\sann_R(\CM(R))$ agree with those that appear in the intersection $\sann_R(\Omega^n \CM(R))$ for any $n \geq 0$.  
\end{proof}
\begin{remark}\label{Kobayashi-Takahashi-remark}
    In \cite{Kobayashi-Takahashi}, Kobayashi and Takahashi studies Ulrich modules over Cohen-Macaulay local rings with minimal multiplicity. Let $(R, \m, k)$ be a $d$-dimensional Cohen-Macaulay local ring with minimal multiplicity. They show that the category $\Omega \CM^{\times}(R)$ is contained in the category $\mathrm{Ul}(R)$ of Ulrich $R$-modules. Then, they show that the equality occurs exactly when our condition is satisfied \cite[Theorem B]{Kobayashi-Takahashi}. More precisely, they show that the following two conditions are equivalent for a Cohen-Macaulay local ring with minimal multiplicity: 
    \begin{enumerate}
        \item The category $\Omega \CM^\times(R)$ is closed under the canonical dual $D(-)$.
        \item There is an equality $\Omega \CM^\times (R) = \mathrm{Ul(R)}$.   
    \end{enumerate}
    They also show that when $k$ infinite and $d$ is positive, these equivalent conditions imply that $R$ is almost Gorenstein (See Definition \ref{almost-Gorenstein-definition}). They also consider the condition 
    \begin{enumerate}
        \setcounter{enumi}{2}
        \item There is an equality $\sann_R(D\Omega^d k) = \m$. 
    \end{enumerate}
    They show that (3) implies (2) (and hence (1)). Below we will show that in fact (3) is equivalent to (2) (and hence (1)). They show this equivalence for $d=1$ with $k$ infinite: in this case, the three equivalent conditions are satisfied if and only if $R$ is almost Gorenstein. They also show the equivalence of (1) and (3) for certain surface singularities.
\end{remark}
The following example follows from the work of Kobayashi and Takahashi \cite[Example 3.12]{Kobayashi-Takahashi}.
\begin{example}
    Let $S = \CC[[x,y,z]]$ be a formal power series ring. Let $G$ be the cyclic group $\dfrac{1}{3}(1,1,1)$. Then, the invariant ring $R = S^G$ (the second Veronese subring of $S$) has three indecomposable maximal Cohen-Macaulay modules up to isomorphism: $R, \omega, \Omega \omega$. Since $R$ is singular, we have two options for $\ca(R)$. It is either $\sann_R(\omega)$ or $\sann_R(\Omega \omega)$. However, by \cite[Example 3.12]{Kobayashi-Takahashi}, we know that $\mathrm{Ul}(R) = \Omega \CM^\times(R)$. Hence, $\Omega \CM^\times(R)$ is closed under $D(-)$. Hence, we actually have 
    \begin{align*}
        \ca(R) = \sann_R(\omega) = \sann_R(\Omega \omega).
    \end{align*}        
    Note that $\sann_R(\omega)$ is the canonical trace ideal (see the next section) and $R$ is nearly Gorenstein. So, we conclude $\ca(R) = \m$.
\end{example}
We now prove the equivalence of (1) and (3) in Remark \ref{Kobayashi-Takahashi-remark}.
\begin{proposition}
    Let $R$ be a $d$-dimensional Cohen-Macaulay singular local ring with minimal multiplicity and with canonical module $\omega$. Then, the following are equivalent.
    \begin{enumerate}
        \item There is an equality $\sann_R(D\Omega^d k) = \m$.
        \item The category $\Omega \CM^\times(R)$ is closed under the canonical dual $D(-)$.
        \item The module $D \Omega^d k$ belongs to $\Omega \CM(R)$.
    \end{enumerate}
\end{proposition}
\begin{proof}
    The implication (1) $\implies$ (2) is from \cite[Corollary 3.11]{Kobayashi-Takahashi} as previously mentioned. The implication (2) $\implies$ (3) follows from the fact that we have inclusions
    \begin{align*}
    \Omega^d k \in \mathrm{TF}_{d+1}(R) \subseteq \Omega^{d+1}{\module R} \subseteq \Omega \CM(R)
    \end{align*}
    where $\mathrm{TF}_{d+1}(R)$ is the category of $(d+1)$-torsionfree modules. In full generality, this was proved by Dey-Takahashi \cite[Theorem 4.1(2)]{Dey-Takahashi}. Now assume that $D\Omega^d k \in \Omega\CM(R)$. Then, by the proof of Proposition \ref{main-proposition}, we have equalities 
    \begin{align*}
    \sann_R(D \Omega^d k) = \sann_R(\Omega^d k) = \m.
    \end{align*}
     
\end{proof}
This discussion regarding rings with minimal multiplicity gives us the following corollary.
\begin{corollary}
    Let $R$ be a Cohen-Macaulay local ring with canonical module $\omega$. Assume that $R$ has Krull dimension $d$ and that it has minimal multiplicity. If $D \Omega^d k \in \Omega \CM(R)$, then we have 
    \begin{align*}
    \ca(R) = \sann_R(\CM(R)).
    \end{align*}
\end{corollary}

We finish this section with another application. We remark that the intersection on the right hand side of the equality in the proposition below is a \textit{cocohomological} annihilator in the sense of Dey-Liu-Mifune-Otake \cite{Dey-Liu-Mifune-Otake}.
\begin{proposition}
    Let $R$ be Cohen-Macaulay local ring with canonical module $\omega$ and $M$ be a maximal Cohen-Macaulay $R$-module. If both $M$ and $DM$ belong to $\Omega \CM(R)$, then there exists an equality
    \begin{align*}
        \bigcap_{X \in \CM(R)} \ann_R \Ext_R^{>0}(M,X) = \bigcap_{X \in \CM(R)} \ann_R \Ext_R^{>0}(X,M)
    \end{align*}
    of ideals.
\end{proposition}
\begin{proof}
    We note that as $M \in \CM(R)$, we also have that $\Omega_RM \in \CM(R)$. Hence, there is an inclusion 
    \begin{align*}
    \bigcap_{X \in \CM(R)} \ann_R \Ext_R^{>0}(M,X) \subseteq \ann_R\Ext_R^1(M, \Omega_R M) = \sann_R(M).
    \end{align*}
    On the other hand as $\CM(R) \subseteq \module(R)$, we have another inclusion 
    \begin{align*}
    \sann_R(M) = \bigcap_{X \in \module(R)} \ann_R \Ext_R^{>0}(M,X) \subseteq \bigcap_{X \in \CM(R)} \ann_R \Ext_R^{>0}(M,X)
    \end{align*}
    which means that the intersection on the left hand side of the equality in the statement is just the stable annihilator of $M$. 
    
    Since $D$ is a duality on $\CM(R)$, we also have equalities
    \begin{align*}
        \bigcap_{X \in \CM(R)} \ann_R \Ext_R^{>0}(X,M) = \bigcap_{X \in \CM(R)} \ann_R \Ext_R^{>0}(DX,M) = \bigcap_{X \in \CM(R)} \ann_R \Ext_R^{>0}(DM,X).
    \end{align*}
    Hence, the right hand side of the equality in the statement is $\sann_R(DM)$. Now, we are done by the proof of Proposition \ref{main-proposition}.
\end{proof}

\section{When is the dual of the syzygy of an MCM module the syzygy of an MCM module?}

The main ingredient of the main theorem in Section 2 was the fact that if a module $M$ and its canonical dual $DM$ are both syzygies of maximal Cohen-Macaulay modules over a Cohen-Macaulay local ring with canonical module $\omega$, then one has an equality $\sann_R(M) = \sann_R(DM)$ of stable annihilators. In this section, we will investigate this property. Our main tool will be trace ideals which we recall next.
\begin{definition}
    Let $R$ be a commutative Noetherian ring and $M$ be a finitely generated $R$-module. Then, the \textit{trace ideal} $\tr(M)$ of $M$ is the ideal of $R$ generated by the elements of the form $f(m)$ where $f \in M^*$ and $m \in M$. 
\end{definition}
We recall some properties of trace ideals \cite[Proposition 2.8]{Lindo-centers}.
\begin{remark}\label{remark-trace}
    Let $R$ be a commutative Noetherian local ring and $M, N$ finitely generated $R$-modules.
    \begin{enumerate}
        \item One has $\tr(M) = R$ if and only if $M$ contains a nonzero direct summand.
        \item There is an inclusion $\tr(N) \subseteq \tr(M)$ if $M$ generates $N$ (that is, if there exists a surjection $M^{\oplus m} \twoheadrightarrow N$).  
    \end{enumerate}
\end{remark}
We start with recording a straightforward observation.
\begin{lemma}\label{lemma-where-the-trace-is-contained-in-canonical-trace}
    Let $R$ be a Cohen-Macaulay local ring with canonical module $\omega$ and $M$ be a maximal Cohen-Macaulay $R$-module such that $DM \in \Omega \CM(R)$. Then, there exists an inclusion $\tr(M) \subseteq \tr(\omega)$.  
\end{lemma}
\begin{proof}
    There is a short exact sequence 
    \begin{align*}
    0 \to DM \to F \to X \to 0
    \end{align*}
    where $F$ is a free $R$-module. Dualising, we get a surjection 
    \begin{align*}
    \omega^{\oplus n} \to M \to 0
    \end{align*}
    for some $n \geq 1$. Therefore, by Remark \ref{remark-trace}(2), we get an inclusion $\tr(M) \subseteq \tr(\omega)$.
\end{proof}
With this lemma, it is natural to ask whether the converse holds. More precisely, we have the following question.
\begin{question}
    Let $R$ be a Cohen-Macaulay local ring with canonical module $\omega$ and $M$ be a maximal Cohen-Macaulay $R$-module such that $\tr(M) \subseteq \tr(\omega)$. Does it follow that the canonical dual $DM$ belongs to $\Omega \CM(R)$? 
\end{question}
We have a positive answer for this in dimension one when the ring has canonical reduction number $2$. As a preparation, we recall some definitions and some known facts.
\begin{definition}
    Let $R$ be a commutative Noetherian local ring and $Q(R)$ its total quotient ring. Denote by $\overline{R}$ the normalisation (i.e. the integral closure) of $R$ inside $Q(R)$. Then, the \textit{conductor} ideal is defined as 
    \begin{align*}
    \co(R) = \{r \in \overline{R} \colon r \overline{R} \subseteq R \}.
    \end{align*}
    More generally, for any birational extension $R \subseteq S$, one can define 
    \begin{align*}
    \co(R,S) = \{r \in S \colon r S \subseteq R \}.
    \end{align*}    
\end{definition}
Here are some facts regarding conductor ideals.
\begin{remark}\label{conductor-remarks}
    Let $R$ be a one dimensional Cohen-Macaulay local ring.
    \begin{enumerate}
        \item The conductor ideal $\co(R)$ is the largest common ideal of $R$ and $\overline{R}$.
        \item The conductor ideal $\co(R)$ is finitely generated as an $R$-module if and only if $R$ is analytically unramified (that is, the completion of $R$ is reduced). In this case, $R$ is itself reduced and it possesses a canonical ideal (an ideal which is isomorphic to a canonical module).
        \item The conductor ideal $\co(R,S)$ is equal to the stable annihilator ideal $\sann_R(S)$ when $S$ is a birational extension which is finitely generated as an $R$-module. A proof of this was given in \cite[Corollary 4.2]{Esentepe} with $S=\overline{R}$ in the case $R$ is Gorenstein but the proof works verbatim in our generality.
        \item As mentioned in the Introduction, Wang proved in \cite[Section 3]{Wang} that the conductor ideal $\co(R)$ annihilates $\Ext_R^1(M,N)$ for any two maximal Cohen-Macaulay $R$-modules $M$ and $N$ provided that $R$ is reduced and complete.
    \end{enumerate}
\end{remark}
\begin{definition}
    Let $R$ be a commutative Noetherian ring and $I$ a regular ideal. Then, there is a filtration of endomorphism algebras 
    \begin{align*}
    R \subseteq I \colon I \subseteq I^2 \colon I^2 \subseteq \cdots I^n \colon I^n \subseteq \cdots \subseteq Q(A).
    \end{align*}
    We put 
    \begin{align*}
    R^I = \bigcup_{n \geq 0} I^n \colon I^n
    \end{align*}        
    and denote by $b(I)$ the conductor ideal $\co(R,R^I)$.
\end{definition}
\begin{remark}\label{blow-up-remarks}
    Assume that $R$ is a commutative Noetherian ring of dimension one and $I$ a regular ideal.
    \begin{enumerate}
        \item The ring $R^I$ coincides with the blowup algebra $B(I)$ at $I$. This was proved by Lipman in his influencial work \cite{Lipman} on Arf rings.
        \item The ideal $b(I)$ equals the trace ideal $\tr_R(R^I)$. This follows from \cite[Proposition 2.4]{Kobayashi-Takahashi-Trace}.
        \item There exists an inclusion $b(I) \subseteq \tr(I)$. The equality holds if and only if there exists an isomorphism $\tr(I) \cong I^*$. This was proved by Dao-Lindo in \cite[Proposition 4.6 and Proposition 4.7]{Dao-Lindo}.
    \end{enumerate}
\end{remark}
Next, we will talk about $I$-Ulrich modules introduced by Dao-Maitra-Sridhar \cite{Dao-Maitra-Sridhar}. Ulrich modules have been an important topic in commutative algebra over four decades now. They are maximally generated maximal Cohen-Macaulay modules and they are named after Bernd Ulrich who initiated theire systematical study in \cite{Ulrich}. The following definition is an equivalent condition for being an $I$-Ulrich module \cite[Theorem 4.6]{Dao-Maitra-Sridhar}. 
\begin{definition}
    Let $R$ be a one dimensional Cohen-Macaulay local ring and $I$ a regular ideal. Then a maximal Cohen-Macaulay $R$-module is called \textit{$I$-Ulrich} if there is an isomorphism $IM \cong M$.
\end{definition}
\begin{remark}\label{i-ulrich-modules}
    Let $R$ be a one dimensional Cohen-Macaulay local ring and $I$ be a regular ideal. 
    \begin{enumerate}
        \item Assume that $S$ is a birational extension of $R$ and $M$ a finitely generated $R$-module. If $M$ is a module over $S$, that it the action of $R$ on $M$ extends to an action of $S$, then there exists an inclusion $\tr(M) \subseteq \tr(S)$. The converse holds when $M$ is reflexive. In fact, this statement holds for any Noetherian ring \cite[Theorem 2.9]{Dao-Maitra-Sridhar}.
        \item A module $M$ is $I$-Ulrich if and only if $M \in \CM(B(I))$ \cite[Theorem 4.6]{Dao-Maitra-Sridhar}. This has the consequence that if $M$ is an $I$-Ulrich module, then one has an inclusion $\tr(M) \subseteq b(I)$ and the converse holds when $M$ is reflexive.
        \item If $M$ is an $I$-Ulrich module, then for any $X \in \CM(R)$, the module $\Hom_R(M,X)$ is also an $I$-Ulrich module \cite[Lemma 4.15]{Dao-Maitra-Sridhar}.
        \item If $R$ is analytically unramified, then the normalisation $\overline{R}$ is $I$-Ulrich.
    \end{enumerate}
\end{remark}
Recall that a Cohen-Macaulay local ring $R$ is generically Gorenstein if for every minimal prime $p$ of $R$ the local ring $R_p$ is Gorenstein. In this case, the ring $R$ admits a canonical module $\omega$. In fact, the generically Gorenstein condition is equivalent to saying that the canonical module $\omega$ is isomorphic to an ideal. When we take $\omega$ to be an ideal we will refer to it as a canonical ideal. 
\begin{remark}\label{omega-Ulrich-remarks}
    Let $R$ be a one dimensional generically Gorenstein local ring with canonical ideal $\omega$.
    \begin{enumerate}
        \item There is an equality $\reflexive(R) = \Omega \CM(R)$. That is, an $R$-module $M$ is reflexive if and only if it is the syzygy of a maximal Cohen-Macaulay module.
        \item An $R$-module $M$ is $\omega$-Ulrich if and only if there is an isomorphism $M^* \cong DM$ \cite[Corollary 4.27]{Dao-Maitra-Sridhar}.
        \item For $n \gg 0$, the ideal $\omega^n$ is $\omega$-Ulrich. 
        \item The maximal ideal $\m$ is $\omega$-Ulrich if and only if $R$ is almost Gorenstein (See Definition \ref{almost-Gorenstein-definition}). This was stated in \cite[Proposition 7.2]{Dao-Maitra-Sridhar} but a proof also follows from the proof of \cite[Theorem 3.12]{Kumashiro}.
    \end{enumerate}
\end{remark}
\begin{proposition}\label{proposition-omega-ulrich}
    Let $R$ be a one dimensional Cohen-Macaulay local ring.
    \begin{enumerate}
        \item Assume that $M$ is $\omega$-Ulrich. Then, its dual $DM$ is in $\Omega\CM(R)$.
        \item Assume that $\Omega M$ is $\omega$-Ulrich. Then, there is an equality 
        \begin{align*}
        \sann_R(M) = \sann_R(\Omega M).
        \end{align*}
    \end{enumerate}
\end{proposition}
\begin{proof}
    The first assertion follows from Remark \ref{omega-Ulrich-remarks}(2). Indeed, $M^*$ is a second syzygy (more precisely, the second syzygy of the Auslander transpose of $M$) and hence the syzygy of a maximal Cohen-Macaulay module. As $DM$ is isomorphic to $M^*$ and $\Omega \CM(R)$ is closed under isomorphisms, we are done. The second assertion then follows from Proposition \ref{main-proposition}.
\end{proof}

The following definition is due to Kumashiro \cite[Definition 2.6]{Kumashiro}.
\begin{definition}
    Let $R$ be a generically Gorenstein Cohen-Macaulay local ring. The infimum of all nonnegative integers $n$ such that there exists a canonical ideal $\omega$ and an almost reduction $(a)$ of $\omega$ with $\omega^{n+1} = a\omega^n$ is called the \textit{canonical reduction number} of $R$ and is denoted by $\mathrm{can.red}(R)$.   
\end{definition}
\begin{remark}\label{canonical-reduction-remarks}
    Let $R$ be a generically Gorenstein Cohen-Macaulay local ring with canonical ideal $\omega$.
    \begin{enumerate}
        \item One has $\mathrm{can.red}(R) \leq 1$ if and only if $R$ is Gorenstein \cite[Proposition 2.9]{Kumashiro}.
        \item One has $\mathrm{can.red}(R) \leq 2$ if and only if there is an isomorphism $\tr(\omega) \cong \omega^*$ \cite[Theorem 2.13]{Kumashiro}. 
        \item If $R$ is one dimensional, then $\mathrm{can.red}(R)$ is bounded above by $e(R) - 1$ where $e(R)$ denotes the multiplicity of $R$.
    \end{enumerate}
\end{remark}
Now, we are ready to state the main theorem of this section.
\begin{theorem}\label{theorem-for-canonical-reduction-2}
    Let $R$ be a one dimensional generically Gorenstein Cohen-Macaulay local ring with canonical ideal $\omega$. Assume that $\mathrm{can.red}(R) \leq 2$. Then, the following are equivalent for $M \in \Omega \CM(R)$.
    \begin{enumerate}
        \item The dual $DM$ belongs to $\Omega \CM(R)$.
        \item There is an inclusion $\tr(M) \subseteq \tr(\omega)$. 
    \end{enumerate}
\end{theorem}
\begin{proof}
    One implication was proved in Lemma \ref{lemma-where-the-trace-is-contained-in-canonical-trace}. For the other implication, assume that $\tr(M) \subseteq \tr(\omega)$. Since $R$ has canonical reduction number at most $2$, by Remark \ref{canonical-reduction-remarks}(2) we have an isomorphism $\tr(\omega) \cong \omega^*$. By Remark \ref{blow-up-remarks}(3), this means that we have an equality $\tr(M) \subseteq b(\omega)$. As $M \in \Omega \CM(R)$, it is reflexive. Therefore, the inclusion $\tr(M) \subseteq b(\omega)$ implies that $M$ is an $\omega$-Ulrich module by Remark \ref{i-ulrich-modules}(2). Hence, by Proposition \ref{proposition-omega-ulrich}, we conclude that $DM \in \Omega \CM(R)$. 
\end{proof}
We give three examples. Before the first one, we recall the definition of an almost Gorenstein local ring. This definition is due to Goto-Takahashi-Taniguchi \cite{Goto-Takahashi-Taniguchi} who extended the definitions of Barucci-Fr{\"o}berg \cite{Barucci-Froberg} and Goto-Matsuoka-Phuong \cite{Goto-Matsuoka-Phuong}.
\begin{definition}\label{almost-Gorenstein-definition}
    Let $R$ be a Cohen-Macaulay local ring with canonical module $\omega$. Then, $R$ is called \textit{almost Gorenstein} if there exists an exact sequence 
    \begin{align*}
    0 \to R \to \omega \to C \to 0
    \end{align*}
    of finitely generated $R$-modules such that the multiplicity of $C$ equals the minimal number of generators of $C$.
\end{definition}
\begin{example}\label{almost-Gorenstein-example-which-shows-that-minimal-multiplicity-can-be-dropped}
    Let $R$ be an almost Gorenstein local ring of dimension one which is not Gorenstein. Then, by \cite[Therem 3.12]{Kumashiro}, $R$ has canonical reduction number 2. The equality $b(\omega) = \tr(\omega) = \m$ implies that every module $M$ in $\Omega \CM^\times(R)$ is $\omega$-Ulrich and hence satisfies an isomorphism $M^* \cong DM$. Moreover, $\Omega \CM^\times(R)$ is closed under canonical dual.
\end{example}
\begin{example}\label{far-flung-example}
    Let $R$ be a one dimensional Cohen-Macaulay local ring with canonical module $\omega$. Assume that $R$ is analytically unramified. This implies that $R$ is reduced and hence generically Gorenstein. In this case, the conductor ideal $\co(R)$ is contained in $\tr(\omega)$. In fact, it is contained in $\tr(M)$ for any $R$-module $M$ provided that $\tr(M)$ is regular \cite[Corollary 3.6]{Dao-Maitra-Sridhar}. Herzog-Kumashiro-Stamate studies in \cite{Herzog-Kumashiro-Stamate} the case where the equality $\co(R)= \tr(\omega)$  holds. They call such rings \textit{far-flung Gorenstein}. Note that, the general inclusions $\co(R) \subseteq b(\omega) \subseteq \tr(\omega)$ imply that far-flung Gorenstein rings have canonical reduction number at most 2. Hence, we are in the landscape of Theorem $\ref{theorem-for-canonical-reduction-2}$. However, assume that $M$ and $DM$ both belong to $\Omega \CM(R)$ and assume that $M$ has constant rank $r >0$. Then, by \cite[Theorem 4.1]{Herzog-Kumashiro-Stamate}, we have isomorphisms
    \begin{align*}
       \Hom_R(DM, R) \cong \Hom_R(\omega, M) \cong \overline{R}^{\oplus r}
    \end{align*}
    and this gives us $DM \cong \overline{R}^{\oplus r}$ since $DM$ is reflexive and $\overline{R}$ is isomorphic to its $R$-dual. Further, since $\overline{R}$ is regular, it is Gorenstein and it is therefore isomorphic, as an $R$-module, to its canonical dual. Therefore, over a far-flung Gorenstein domain, if $M$ and $DM$ are both reflexive, then $M \cong \overline{R}^{\oplus R}$ for some $R$. 
\end{example}
\begin{example}\label{multiplicity-3-example}
    Let $R$ be a one dimensional Cohen-Macaulay local ring. Assume that $R$ is generically Gorenstein and it has multiplicity $3$. Then, by Remark \ref{canonical-reduction-remarks}(3), $R$ has canonical reduction number at most 2. So, for rings with multiplicity 3, when $M \in \Omega \CM(R)$, the dual $DM$ also is in $\Omega \CM(R)$ if and only if $\tr(M) \subseteq \tr(\omega)$. 
\end{example}
We are now ready to compute the cohomology annihilator ideal of a one dimensional almost Gorenstein local ring.
\begin{theorem}\label{theorem-cohomology-annihilator-of-an-almost-Gorenstein-ring}
    Let $R$ be an analytically unramified one dimensional almost Gorenstein complete local ring. Then, we have an equality 
    \begin{align*}
    \ca(R) = \co(R).
    \end{align*}
\end{theorem}
\begin{proof}
    By Example \ref{almost-Gorenstein-example-which-shows-that-minimal-multiplicity-can-be-dropped}, we know that $\Omega \CM^\times(R)$ is closed under the canonical dual. Therefore, by Theorem \ref{main-theorem}, we have an equality 
    \begin{align*}
    \ca(R) = \sann_R(\CM(R)).
    \end{align*}
    Hence, we have $\ca(R) \subseteq \sann_R(\overline{R}) = \co(R)$. On the other hand, by Wang's result which was mentioned in Remark \ref{conductor-remarks}(4), we have $\co(R) \subseteq \ca(R)$. Combining the two results, we obtain the desired equality.  
\end{proof}

\bibliographystyle{alpha}
\bibliography{ca_ag.bib}
\end{document}